\documentclass[a4paper, reqno, 10pt]{amsart}

\usepackage{amsmath, amsfonts, amssymb, amsthm, graphics, graphicx}
\usepackage{hyperref}
\usepackage{enumitem}

\newtheorem{thm}{Theorem}[section]
\newtheorem{lemma}[thm]{Lemma}

\newtheorem{question}[thm]{Question}

\providecommand{\aut}{\mathop{\rm Aut \,}\nolimits}
\providecommand{\sym}{\mathop{\rm Sym \,}\nolimits}
\providecommand{\soc}{\mathop{\rm soc}\nolimits}
\providecommand{\Wr}{\mathop{\rm Wr}\nolimits}
\providecommand{\twr}{\mathop{\rm twr}\nolimits}
\providecommand{\Inn}{\mathop{\rm Inn}\nolimits}

\renewcommand{\\}{\vspace{3mm}}

\title{\bf Primitive permutation groups whose subdegrees are bounded above}
\author{\bf Simon M. Smith}
\email{simon.smith@chch.oxon.org, smsmit13@syr.edu}
\address{
   Department of Mathematics, Syracuse University \\
   Syracuse, New York, U.S.A.}
\date{\today}
\begin{document}
\maketitle

\markboth{\textsc{Primitive permutation groups with bounded subdegrees}}{\textsc{Primitive permutation groups with bounded subdegrees}}

\begin{abstract} 

If $G$ is a group of permutations of a set $\Omega$ and $\alpha \in \Omega$, then the {\em $\alpha$-suborbits} of $G$ are the orbits of the stabilizer $G_\alpha$ on $\Omega$. The cardinality of an $\alpha$-suborbit is called a {\em subdegree} of $G$. If the only $G$-invariant equivalence classes on $\Omega$ are the trivial and universal relations, then $G$ is said to be a {\em primitive} group of permutations of $\Omega$.

In this paper we determine the structure of all primitive permutation groups whose subdegrees are bounded above by a finite cardinal number.
\end{abstract}

\section{Introduction}

Recall that if $G$ is a group of permutations of a set $\Omega$, and $\alpha \in \Omega$, then the {\em $\alpha$-suborbits} (or sometimes just {\em suborbits}) of $G$ are the orbits of the point stabilizer $G_\alpha$ acting on $\Omega$. The cardinalities of these suborbits are called the {\em subdegrees} of $G$. If $G$ acts transitively on $\Omega$, and $\Omega$ admits no $G$-invariant equivalence relations except the universal relation $\{\Omega\}$ and the trivial relation $\{\{\beta\} : \beta \in \Omega\}$, then $G$ is said to act {\em primitively} on $\Omega$ and $G$ is called a {\em primitive} permutation group. If $G$ is primitive and $N  \trianglelefteq G$, then the orbits of $N$ are the equivalence classes of a $G$-invariant equivalence relation on $\Omega$, and so $N$ is either transitive on $\Omega$ or $N$ is trivial. A {\em minimal normal subgroup} of $G$ is a non-trivial normal subgroup of $G$ that does not properly contain any non-trivial normal subgroup of $G$. We say {\em countable} to mean a set is finite or countably infinite, and {\em denumerable} to mean it is countably infinite.

 Throughout, let $G$ be a primitive group of permutations of a set $\Omega$, with $\alpha$ some fixed element of $\Omega$. If the $\alpha$-subdegrees of $G$ are all finite, then $G$ is said to be {\em subdegree finite}, and it is easily seen that $\Omega$ is countable. Subdegree finite permutation groups are among the most  frequently encountered infinite permutation groups, as all automorphism groups of locally finite combinatorial structures (like graphs) are subdegree finite.

In \cite{me:SubdegreeGrowth}, a broad program of classifying infinite subdegree finite primitive permutation groups according to the growth rate of their subdegrees is advocated. In this paper we continue this program by fully classifying those groups whose subdegrees are bounded above. Our main result is the following. The various types of subdegree finite primitive permutation group that occur in this classification are described in Section~\ref{Types}.

\begin{thm} \label{main_thm} A primitive permutation group $G \leq \sym(\Omega)$ whose subdegrees are bounded above by a finite cardinal lies in precisely one of the following classes.

\begin{description}[labelindent=\parindent, font=\normalfont]
\item[I] Finite affine \label{test}
\item[II] Countable almost simple
\item[III] Product
	\begin{description}[labelindent=\parindent, font=\normalfont]
	\item[III(a)] Finite simple diagonal action
	\item[III(b)] Countable product action
	\item[III(c)] Finite twisted wreath action
	\end{description}
\item[IV] Denumerable split extension
\end{description}
\end{thm}

Finite examples for types I, II, III(a), III(b), III(c), and infinite examples for types II, III(b) and IV exist.

\section{Types of primitive permutation group with bounded subdegrees}
\label{Types}

The finite types that occur in our classification are precisely those found in the O'Nan-Scott Theorem (first appearing in \cite{Scott:Onan_Scott}, with a modern formulation in \cite{liebeck&praeger&saxl:finite_onan_scott}). The infinite types that occur are precisely those found in the classification of the infinite primitive permutation groups with finite point stabilizers (\cite{sms_finite_stabilizers}); refer to \cite{cameron:permutation_groups}, \cite{liebeck&praeger&saxl:finite_onan_scott} and \cite{sms_finite_stabilizers} for detailed descriptions of the various types.  In every case $\Omega$ is a countable set and $\alpha \in \Omega$ is some fixed element. The group $G$ is finitely generated by elements of finite order. The socle $B$ of $G$ is the subgroup generated by all the minimal normal subgroups of $G$; there exists a non-trivial finitely generated simple group $K$ such that $B = K_1 \times \cdots \times K_m$, where $m\geq 1$ is finite and $K_i \cong K$ for $1 \leq i \leq m$. The stabilizer $G_\alpha$ is always finite, and when $G$ is infinite $B$ is the unique minimal normal subgroup of $G$.

\\

\begin{description}[font=\normalfont]
\item[I] {\em Finite affine.}
	In this case $K \cong Z_p$ for some prime $p$ and $B$ acts regularly on 	$\Omega$, so $|\Omega| = p^m$. We identify the set $\Omega$ with $B = Z_p^m$ 	so that $G$ is permutation isomorphic to a subgroup $H$ of $AGL(m, p)$, with $B$ the translation group and $G_\alpha$ identified with $H \cap GL(m, p)$. Furthermore, $H \cap GL(m, p)$ acts irreducibly on $B$.

\\

\item[II]{\em Countable almost simple.} \label{type:almost_simple}
	Here $m = 1$, $K$ is a non-regular non-abelian finitely generated simple group of finite index in 	$G$ and $K \leq G \leq \aut K$, with $B = K$.

\\

\item[III] {\em Product.}
	Here $m \geq 2$ and $K$ is non-abelian, finitely generated and simple. As is 	traditional, we subdivide this type into three distinct subtypes.

\\

	\begin{description}[font=\normalfont]
	\item[III(a)] {\em Finite simple diagonal action.} Here $G$ is finite. Define $D = \{ (k, k, \ldots, k) : k \in K\}$, and note that $D \leq K^m \cong B$. Let $W = K^m.(\text{Out } K \times S_m)$ be a (not necessarily split) extension of $K^m$ by $\text{Out } K \times S_m$. The group $D$ is called the diagonal subgroup, and there is an obvious action of $W$ on the set of (right) cosets of $D$ in $K^m$. We may take as a set $\Delta$ of coset representatives all those elements of $K^m$ which have the identity in their first coordinate. In this way, we have an action of $W$ on $\Delta$, with $|\Delta| = |K|^{m-1}$, and the stabilizer in $W$ of the representative of the coset $D.1$ is $\{(a, \ldots, a).\pi : a \in \aut K, \pi \in S_m\}$. We say $G$ is of diagonal type if $(G, \Omega)$ is permutation isomorphic to $(H, \Delta)$, where $K^m \leq H \leq W$.
	
	Identifying $(G, \Omega)$ and $(H, \Delta)$, we note that $G$ is primitive if and only if the subgroup $\overline{G}$ of $S_m$ induced by $G$ acting on the direct factors $T = \{K_1, \ldots, K_m\}$ of the socle $B$ preserves no non-trivial congruence. Thus, either (i) $\overline{G}$ is primitive on $T$, or (ii) $m=2$ and $\overline{G} = 1$.
	
\\

	\item[III(b)] {\em Countable product action.} Here for some finite $l > 1$ and some primitive permutation group $(H, \Gamma)$ we have $(G, \Omega)$ is permutation isomorphic to a subgroup of $(H \Wr S_l, \Gamma^l)$ acting via the product action, where, for some $\gamma \in \Gamma$, the stabilizer $G_\alpha$ is identified with a subgroup of $H_\gamma \Wr S_l$ and one of the following holds:
\begin{enumerate}
\item
	$H$ is countable and of type II, with $H_\gamma$ finite, $\soc(H) = K$ and $l=m$; or
\item
	$H$ is finite and of type III(a), with $\soc(H) = K^{m/l}$, and $G$ and $H$ both have at most two minimal normal subgroups.
\end{enumerate}

\\

	\item[III(c)] {\em Finite twisted wreath action.} Here $G$ and hence $K$ are finite, and $B \cong K^m$ acts regularly on $\Omega$. In this case there exists a transitive subgroup $P$ of $S_m$ such that $(G, \Omega)$ is permutation isomorphic to $(K \twr_\varphi P, \ \Omega)$, where $K \twr_\varphi P$ is a twisted wreath product, described as follows.
Recall $P$ acts transitively on $\{1, 2, \ldots, m\}$; let $Q$ be the stabilizer of $1$ in $P$. Suppose there is a homomorphism $\varphi: Q \rightarrow \aut K$ such that $\text{Im} (\varphi)$ contains $\Inn (K)$. Let
\[A = \{f: P \rightarrow K :  f(pq) = f(p)^{\varphi(q)} \ \ \forall p \in P, \ \forall q \in Q\}.\]

Now $A$ is a group under pointwise multiplication and $A \cong K^m \cong B$. The transitive group $P$ acts on $A$ in a natural way, with $f^p(x) = f(px)$ for all $p, x \in P$. The twisted wreath product $K \twr_\varphi P$ is then the semidirect product of $A$ with $P$, and the action of $K \twr_\varphi P$ on $\Omega$ is determined by taking $P$ to be the stabilizer of $\alpha$ in $K \twr_\varphi P$.
	\end{description}

\\

\item[IV] {\em Denumerable split extension.}
Groups $G$ and $K$ are infinite, and $B$ acts regularly on $\Omega$. Here $G$ is equal to the split extension $M.G_\alpha$ for some $\alpha \in \Omega$, and no non-identity element of $G_\alpha$ induces an inner automorphism of $M$.
\end{description}

\section{Primitive permutation groups with bounded subdegrees}

The following result is commonly known as Schlichting's Theorem (\cite{schlichting}); it was proved independently by Bergman and Lenstra (\cite{bergman_lenstra}).

\begin{thm} \label{bl}
Let $G$ be a group and $H$ a subgroup. Then the following conditions are equivalent:
\begin{enumerate}
\item
	the set of indices $\{ |H: H \cap gHg^{-1} | : g \in G \}$ has a finite upper bound;
\item
	there exists a normal subgroup $N  \trianglelefteq G$ such that both $|H : H \cap N|$ and $|N : H \cap N|$ are finite. \qed
\end{enumerate}
\end{thm}

From this we obtain the following.

\begin{lemma} \label{lemma:bd_subd_implies_finite_stab} Suppose $G$ is a primitive group of permutations of an infinite set  $\Omega$, and $\alpha \in \Omega$. The $\alpha$-subdegrees of $G$ are bounded above by a finite cardinal if and only if $G_\alpha$ is a finite permutation group.
\end{lemma}

\begin{proof} Suppose $\Omega$ is an infinite set, and fix $\alpha \in \Omega$. Because $G$ is transitive, if $\beta \in \Omega$ then there exists $g \in G$ such that $\alpha^g = \beta$. Thus, for all $\beta \in \Omega$, the length of the suborbit $\beta^{G_\alpha}$ is $|G_\alpha : G_{\alpha, \beta}| = |G_\alpha : G_\alpha \cap g^{-1} G_\alpha g|$ for some $g \in G$. If the subdegrees of $G$ are bounded above by a finite cardinal, then so are the elements of the set $\{|G_\alpha : G_\alpha \cap g^{-1} G_\alpha g| : g \in G\}$. By Theorem~\ref{bl}, there exists $N  \trianglelefteq G$ such that $|G_\alpha : G_\alpha \cap N|$ and $|N : G_\alpha \cap N|$ are finite. Because $G$ is primitive, if $N$ is non-trivial then $N$ must act transitively on $\Omega$, and $|N : G_\alpha \cap N| = |\alpha^N| = |\Omega|$. Hence $N$ is trivial and $G_\alpha$ is finite.

The converse follows immediately from the observation that every subdegree of $G$ is bounded above by $|G_\alpha |$.
\end{proof}

If a point stabilizer in $G$ is finite then all point stabilizers in $G$ are finite. Infinite primitive permutation groups with finite point stabilizers were classified in \cite{sms_finite_stabilizers}.

\begin{thm}[\cite{sms_finite_stabilizers}] An infinite primitive permutation group with a finite point-stabilizer lies in precisely one of the following classes: II (infinite only), III(b) (infinite only) and IV. Examples exist for each type. \qed
\end{thm}

The finite primitive permutation groups were classified by the O'Nan-Scott Theorem.

\begin{thm}[\cite{liebeck&praeger&saxl:finite_onan_scott}] Any finite primitive permutation group lies in precisely one of the following classes: I, II (finite only), III(a), III(b) (finite only) and III(c). \qed
\end{thm}

Combining these two results we obtain the following classification of primitive permutation groups with finite point stabilizers.

\begin{thm} Any primitive permutation group with a finite point stabilizer lies in precisely one of the following classes: I, II, III(a), III(b), III(c) and IV. \qed
\end{thm}

This theorem, in conjunction with Lemma~\ref{lemma:bd_subd_implies_finite_stab}, establishes our main result, Theorem~\ref{main_thm}. Examples exist for each type (see \cite{dixon&mortimer} and \cite{sms_finite_stabilizers}).

\\

This work raises some interesting questions.

\begin{question} Which finite groups occur as point stabilizers of infinite primitive permutation groups? This is of course equivalent to asking which finite groups occur as point stabilizers of infinite primitive permutation groups whose subdegrees are bounded above by a finite cardinal.
\end{question}

\begin{question} For which finite sequences $n_1 < \cdots < n_m$ of natural numbers does there exist a primitive permutation group whose set of subdegrees is precisely $\{n_1, \ldots, n_m\}$?
\end{question}

\begin{question} If $G$ is a primitive permutation group whose subdegrees are bounded above by some finite cardinal $n$, what can be said of $G$ if one knows $n$?
\end{question}

%
%


\begin{thebibliography}{99}
\markright{}
%
\bibitem {bergman_lenstra} G. M. Bergman and H. W. Lenstra,
    `Subgroups close to normal subgroups',
    {\em J. Algebra}
    127 (1989) 80--97.
%
\bibitem {cameron:permutation_groups} P. J. Cameron,
    {\em Permutation groups},
    London Mathematical Society Student Texts 45
    (Cambridge University Press, Cambridge, 1999).
%
\bibitem {dixon&mortimer} J. Dixon and B. Mortimer,
    {\em Permutation groups},
    Graduate Texts in Mathematics 163
    (Springer-Verlag, New York, 1996).
%
\bibitem {liebeck&praeger&saxl:finite_onan_scott} M. W. Liebeck, C. E. Praeger, and J. Saxl,
    `On the O'Nan--Scott Theorem for finite primitive permutation groups',
    {\em J. Austral. Math. Soc.}
    44 (1988) 389--396.
%
\bibitem {schlichting} G. Schlichting,
	`Operationen mit periodischen Stabilisatoren',
	{\em Arch. Math. (Basel)}
	34 (1980) 97--99.
%
\bibitem {Scott:Onan_Scott} L. Scott,
    `Representations in characteristic $p$',
    {\em Santa Cruz conference on finite groups},
    Proceedings of Symposia in Pure Mathematics 37
    (American Mathematical Society, Providence, R.I., 1980), 318--331.
%
\bibitem {sms_finite_stabilizers} S. M. Smith,
	`A classification of primitive permutation groups with finite stabilizers',
	 {\em to appear},
	 arXiv:1109.5432v1.
%
\bibitem {me:SubdegreeGrowth} S. M. Smith,
    `Subdegree growth rates of infinite primitive permutation groups',
    {\em J. London Math. Soc.}
    82 (2010) 526--548.
%
\end{thebibliography}
\end{document}